\documentclass{amsart}
\usepackage[a4paper,margin=2.1cm]{geometry}
\usepackage{amssymb,esint}
\usepackage{amscd}
\usepackage{xspace}
\usepackage{fancyhdr}
\usepackage{xcolor}
\usepackage{hyperref,amsmath}
\usepackage{cite}
\usepackage{stmaryrd}

\newtheorem{theorem}{Theorem}[section]
\newtheorem{lemma}[theorem]{Lemma}

\theoremstyle{definition}

\theoremstyle{proposition}
\newtheorem{proposition}[theorem]{Proposition}
\theoremstyle{remark}
\newtheorem{remark}[theorem]{Remark}

\numberwithin{equation}{section}

\begin{document}

\title{Global smooth solutions to Navier-Stokes equations with large initial data in critical space}

\author{Haina Li}
\address{School of Mathematics and Statistics, Beijing Institute of Technology, Beijing 100081, China.}
\email{lihaina2000@163.com}

\author{Yiran Xu}
\address{Fudan University, 220 Handan Road, Yangpu, Shanghai, 200433, China.}
\email{yrxu20@fudan.edu.cn}

%    General info

%\date{ xxxxx and, in revised form, xxxx.}



\begin{abstract}
	In this paper, we investigate the existence of a unique global smooth solution to the three-dimensional incompressible Navier-Stokes equations and provide a concise proof. We establish a new global well-posedness result that allows the initial data to be arbitrarily large within the critical space $\dot{B}^{-1}_{\infty,\infty}$, while still satisfying the nonlinear smallness condition. % introduced in \cite{J.Y. Chemin}.
\end{abstract}	
\maketitle

\section{Introduction}
In this paper, we consider the three-dimensional incompressible Navier-Stokes equations with initial data $u_0\in L^2 (\mathbb{R}^3)$, which can be written as
\begin{equation}\label{NS}
	\begin{cases}
		\partial_t u-\nu\Delta u+\nabla p=-div(u\otimes u),\\
		\nabla\cdot u=0,\\
		u(t,x)|_{t=0}=u_0(x),
	\end{cases}
\end{equation}
where $u:[0,T) \times \mathbb{R}^3 \rightarrow \mathbb{R}^3$ represents the velocity field, and $p:[0,T) \times \mathbb{R}^3 \rightarrow \mathbb{R}$ identifies the pressure. Furthermore, for the sake of simplicity, we assume the kinematic viscosity $\nu=1$ throughout this paper.

Let us give a review of foundational results. In groundbreaking work\cite{Leray}, Leray established the existence of weak solutions to the Navier-Stokes equations \eqref{NS}. This was achieved under the assumption that the initial data $u_0 \in L^2(\mathbb{R}^3)$ satisfies the divergence-free conditions and fulfills the energy inequality:
\begin{equation*}
	\|u(t,\cdot)\|_{L^2(\mathbb{R}^3)}^2
	+2\nu \int_{0}^{t} \|\nabla u(s,\cdot)\|_{L^2(\mathbb{R}^3)}^2 ds
	\leq \|u_0\|_{L^2(\mathbb{R}^3)}^2.
\end{equation*}
The solutions that comply with the aforementioned conditions are called Leray-Hopf weak solutions. Up to now, the issue of their uniqueness continues to be an unresolved problem within the framework of the Millennium Prize Problems\cite{clay}. We refer interested readers to \cite{H. Beirão da Veiga} and \cite{G. Ponce} for relevant problems of the Navier-Stokes equations. It is well established that the Navier-Stokes equations \eqref{NS} possess a global smooth solution when the initial data is sufficiently small within the scale-invariant space $\dot{H}^{\frac{1}{2}}$. Fujita and Kato obtained  the result in $\dot{H}^{\frac{1}{2}}$ in \cite{Fujita and Kato 1964}.

We recall the scaling properties of the equations \eqref{NS}. If $(u, p)$ is a solution to equation, then for any positive constant $\lambda$, the scaled solutions are given by:
\begin{equation*}
	u_\lambda(t, x) = \lambda u(\lambda^2t, \lambda x), \quad p_\lambda(t, x) = \lambda^2p(\lambda^2t, \lambda x).
\end{equation*}
These scaled quantities also satisfy the equation, which corresponds to the rescaled initial data $u_{0,\lambda}(x) = \lambda u_{0}(\lambda x)$. This observation leads us to consider the scale-invariant space.

A number of  works have been dedicated to establishing similar well-posedness results for the Navier-Stokes equations \eqref{NS} with small initial data in critical spaces. T. Kato \cite{T. Kato} initiated the investigation of the Navier-Stokes equations in critical spaces by demonstrating that $\eqref{NS}$ is globally well-posed for small initial data in $L^3(\mathbb{R}^3)$. Giga and Miyakawa\cite{Y. Giga} extended these findings for initial data in certain Morrey spaces. Subsequent work by Taylor \cite{Michael E. Taylor} further generalized the results of both \cite{T. Kato} and \cite{Y. Giga}. Additionally, Cannone \cite{Cannone} proved that Kato's result holds under a much weaker condition on the initial data. A similar conclusion was reached by Planchon  \cite{F. Planchon}, who considered conditions based on norms in Besov spaces.

It is important to highlight the contributions of Koch and Tataru. In \cite{H. Koch}, they established the global well-posedness of the Navier-Stokes equations \eqref{NS} for small initial data in the space $BMO^{-1}$. Subsequently, Auscher, Dubois and Tchamitchian \cite{P. Auscher} demonstrated the stability of global solutions to \eqref{NS} with data in $VMO^{-1}$, where the data belong to the space defined by Koch and Tataru. See also the articles by \cite{T. Miyakawa} and \cite{Y. Meyer}. 

Here we are interested in sufficiently large initial data in critical spaces.
% It is imperative to emphasize the case in which the initial data is sufficiently large within the critical spaces.
Chemin and Gallagher\cite{I. Gallagher} established a novel result concerning global well-posedness, 
under a nonlinear smallness condition imposed on the initial data
\begin{equation*}
	\|\mathbb{P}(e^{t\Delta} u_0 \cdot \nabla e^{t\Delta} u_0)\|_E \leq C^{-1}_0 \exp(-C_0 \|u_0\|_{\dot{B}^{-1}_{\infty,2}}^4).
\end{equation*}
They also presented an example of initial data that meets the aforementioned nonlinear smallness condition, while possessing a norm that can be made arbitrarily large in $\dot{B}^{-1}_{\infty,\infty}$
\begin{equation*}
	u_{0,\epsilon}(x)=(\partial_2 \phi_\epsilon(x),-\partial_1\phi_\epsilon(x),0).
\end{equation*}		
Subsequently, G. Ciampa and R. Lucà \cite{G. Ciampa} developed a class of initial data characterized by arbitrarily large scaling-invariant norms, ensuring that the equations exhibit classical well-posedness. Notably, the initial data adhere to the nonlinear smallness assumption proposed by Chemin and Gallagher in \cite{I. Gallagher}. We also direct the reader's attention to the works of \cite{M. Fujii} and \cite{S. Wang}.

Regarding the hierarchy of critical spaces for the 3D Navier-Stokes equations \cite{M. Cannone}, we have the following embedding relationships: 
\begin{equation*}
	\dot{H}^{\frac{1}{2}}\hookrightarrow  L^3
	\hookrightarrow \dot{B}^{-1+\frac{3}{p}}_{p,\infty} (p<\infty)
	\hookrightarrow BMO^{-1}
	\hookrightarrow  \dot{B}^{-1}_{\infty,\infty}.
\end{equation*}
These embeddings indicate that all critical spaces associated with the 3D Navier-Stokes equations are contained within the same function space $\dot{B}^{-1}_{\infty,\infty}(\mathbb{R}^3)$. Nevertheless, the Navier-Stokes equations are ill-posed in the space $\dot{B}^{-1}_{\infty,\infty}(\mathbb{R}^3)$ \cite{J. Bourgain}. Now the problem concerning large data remains entirely unresolved. Motivated by this, the primary objective of this paper is to establish the existence of global smooth solutions for \eqref{NS}. We allow the initial data to be arbitrarily large within the scale-invariant space $\dot{B}^{-1}_{\infty,\infty}(\mathbb{R}^3)$, provided that it satisfies the nonlinear smallness hypothesis.

Throughout the paper, we adopt the notation:
\begin{equation*}
	u_L=e^{t\Delta} u_0,
\end{equation*}
where $e^{t\Delta}$ denotes the heat flow.

Our main result of this paper states as follows:
\begin{theorem}\label{thm 1.1}
	Let initial data $u_0\in L^2(\mathbb{R}^3)$ satisfy
	$
	\|u_L\|_{L^\infty(\mathbb{R}^3)} 
	\leq M,
	$
	if there exists $\epsilon_0>0$ such that 
	\begin{equation}\label{condition}
		\underset{t\geq 0 }{\sup} \min\{1,t\}^{\frac{1}{8}}\|U_0(t,x)\|_{\dot{H}^{-1}(\mathbb{R}^3)}
		+
		\|U_0\|_{L^2([0,T^*],L_x^2(\mathbb{R}^3))}
		\leq \epsilon_0,
	\end{equation}	
	then the system \eqref{NS} admits  unique global smooth solution, where  $U_0(t,x)=div(u_L \otimes u_L)$.
\end{theorem}	

\begin{remark}
	Regarding inequality \eqref{condition}, we may select any $\gamma\in(0,\frac{1}{4})$. Specifically, this leads to the following condition:
	\begin{equation*}
		\underset{t\geq 0 }{\sup} \min\{1,t\}^{\frac{1}{4}-\gamma}\|U_0(t,x)\|_{\dot{H}^{-1}(\mathbb{R}^3)}
		+
		\|U_0\|_{L^2([0,T^*],L_x^2(\mathbb{R}^3))}
		\leq \epsilon_0.
	\end{equation*}	
	For convenience, we have chosen $ \gamma=\frac{1}{8}$ in this paper.
	Moreover, the initial data in our paper can be arbitrarily large in the critical space $\dot{B}^{-1}_{\infty,\infty}$.
\end{remark}	

\iffalse
\begin{remark}
	We \\
		Moreover, the initial data can be arbitrarily large in $\dot{B}^{-1}_{\infty,\infty}$.
\end{remark}
\fi
\iffalse
\begin{remark}
	\textcolor{red}{Let us give an example where the conditions mentioned in Theorem \ref{thm 1.1} holds. 
	......}
	
	Note that the initial data from \cite{J.Y. Chemin} is the following
	\begin{equation*}
		u_0(x)=(Nv_0^h(x_h)\cos(Nx_3),-div_h v_0^h(x_h)\sin(Nx_3)),
	\end{equation*}	
	where $N_0$ be a given positive integer, $v_0^h$ is any two component satisfying
	\begin{equation*}
		Supp ~\hat{v}_0^h \subset [-N_0,N_0]^2,~ \text{and}~ \|v_0^h\|_{L^2(\mathbb{T}^2)}\leq (\log N)^{\frac{1}{9}}.
	\end{equation*}	
	It is straightforward to demonstrate that the initial data presented in \cite{J.Y. Chemin} fulfill the conditions stipulated in our main theorem. Therefore, we shall omit the detailed proof here.

\end{remark}	
	\fi

\textbf{Organization of the paper.} 
We end this section with the structure of this paper. In section $2$, we established the notations and presented a Lemma that will be utilized in the proof of the theorem. In section $3$, we proved several estimates of the nonlinear terms, which are crucial for the proof of the theorem. Lastly, in section $4$ we presented a short proof of Theorem \ref{thm 1.1}.
\section{Notations and Preliminaries}
Throughout the paper, we denote $A\lesssim B$ if there exists a universal constant C such that $A\leq CB$. To address the estimation of the nonlinear term, it is essential to establish the following fundamental property.
\begin{lemma}\label{Lemma 1}
	Let $A(t)$ be a non-negative continuous function defined on the interval $[0,t]$ that satisfies the following inequality:
	\begin{equation*}
		\begin{split}
		A(t)&
		\leq a_0 
		+C_1t^{\frac{1}{8}} \int_{0}^{t} 
		\frac{1}{\tau^{\frac{1}{8}}(t-\tau)^{\frac{1}{2}}}
		A(\tau) d\tau
		+C_2 t^{\frac{1}{8}}
		\int_{0}^{t} \frac{1}{\tau^{\frac{1}{8}}(t-\tau)^{\frac{7}{8}}} A(\tau) d\tau, \quad 0\leq t\leq1	\\
		A(t)&
		\leq a_0 
		+C_1 \int_{0}^{t} 
		\frac{1}{(t-\tau)^{\frac{1}{2}}}
		A(\tau) d\tau
		+C_2 t^{\frac{1}{8}}
		\int_{0}^{t} \frac{1}{(t-\tau)^{\frac{7}{8}}} A(\tau) d\tau, \quad t\geq1
		\end{split}
	\end{equation*}
	then it holds that
	\begin{equation}\label{A(t)}
		\underset{0\leq t \leq T}{\sup}A(t) \leq e^{20^8T(C_1+10)^8(C_2+10)^8} a_0.
	\end{equation}
\end{lemma}	
\begin{proof}
	To obtain our result, we consider two cases. For the case $0\leq t \leq 1$,	we denote 
	\begin{equation*}
		F_1(t,t_0):=\underset{\tilde{t} \in [t_0,t]}{\sup}
		A(\tilde{t}), \quad F_1(t):=F_1(t,0),
	\end{equation*}
	Consequently, we can obtain
	\begin{equation*}
		\begin{split}
		F_1(t)
		&\leq a_0
		+C_1 t^{\frac{1}{2}} \int_{0}^{1} \lambda^{-\frac{1}{8}} (1-\lambda)^{-\frac{1}{2}}
		F_1(\lambda t) d\lambda
		+
		C_2 t^{\frac{1}{8}} \int_{0}^{1} \lambda^{-\frac{1}{8}} (1-\lambda)^{-\frac{7}{8}}F_1(\lambda t) d\lambda\\
		&\leq  a_0
		+10\big(C_1 t^{\frac{1}{2}} 
		+C_2 t^{\frac{1}{8}}\big) F_1(t).
		\end{split}
	\end{equation*}
	It suffices to establish the existence of $T_0=\frac{1}{20^8(C_1+10)^8(C_2+10)^8}$ such that $10\big(C_1 T_0^{\frac{1}{2}} 
	+C_2 T_0^{\frac{1}{8}}\big) \leq\frac{1}{2}$, which means 
	\begin{equation*}
		F_1(T_0) \leq 2a_0.
	\end{equation*}
	We now address $A(t)$ starting from $T_0$. We write
	\begin{equation*}
		\begin{split}
			F_1(t,T_0)&\leq 2a_0
			+C_1 (t-T_0)^{\frac{1}{2}} \int_{0}^{1} \lambda^{-\frac{1}{8}}(1-\lambda)^{-\frac{1}{2}} F_1(T_0+\lambda (t-T_0)) d\lambda\\
			&\qquad+C_2(t-T_0)^{\frac{1}{8}}  \int_{0}^{1} \lambda^{-\frac{1}{8}}(1-\lambda)^{-\frac{7}{8}}F_1(T_0+\lambda (t-T_0)) d\lambda\\
			&\leq 2a_0
			+10\big[
			C_1 (t-T_0)^{\frac{1}{2}}
			+C_2(t-T_0)^{\frac{1}{8}} 
			\big] F_1(t,T_0).
		\end{split}
	\end{equation*}	
	It is straightforward to derive that
	$
	F_1(2T_0,T_0) \leq 4a_0.
	$
	Without loss of generality, we can conclude that
	\begin{equation}\label{F1}
		F_1(T)\leq 2^{\frac{T}{T_0}} a_0
		\leq e^{20^8T(C_1+10)^8(C_2+10)^8} a_0.
	\end{equation}
	
	Next for the case $t\geq 1$, we similarly define
	\begin{equation*}
		F_2(t,t_0):=\underset{\tilde{t} \in [t_0,t]}{\sup}
		A(\tilde{t}), \quad F_2(t):=F_2(t,0).
	\end{equation*}
	Applying the same method we have
	\begin{equation*}
		F_2(t)
		\leq a_0
		+C_1 t^{\frac{1}{2}} \int_{0}^{1} (1-\lambda)^{-\frac{1}{2}} F_2(\lambda t) d\lambda
		+C_2 t^{\frac{1}{4}} \int_{0}^{1} (1-\lambda)^{-\frac{7}{8}}F_2(\lambda t) d\lambda
		\leq  a_0
		+8\big(C_1 t^{\frac{1}{2}}
		+C_2 t^{\frac{1}{4}}\big) F_2(t).
	\end{equation*}
	Thus, there exists 
	$T_0=\frac{1}{(C_1+10)^4(C_2+10)^4}$ such that $8\big(C_1 {T_0}^{\frac{1}{2}}
	+C_2 {T_0}^{\frac{1}{4}}\big)\leq\frac{1}{2}$, we can arrive at
	\begin{equation*}
		F_2(T_0) \leq 2a_0.
	\end{equation*}
	As a result, adopting similar iterative method, we can infer 
	\begin{equation}\label{F2}
		F_2(T)\leq 2^{\frac{T}{T_0}} a_0
		\leq e^{T(C_1+10)^4(C_2+10)^4} a_0.
	\end{equation}
	Combining \eqref{F1} and \eqref{F2}, we have done the  proof of \eqref{A(t)}.
\end{proof}
\section{Propositions of nonlinear term}
In this section, we present several pertinent propositions that will be utilized in the proof of the Theorem. For convenience, we first introduce the following norm:
\begin{equation*}
	\interleave f(t) \interleave_{p,T}:=\underset{[0,T]}{\sup} \min\{1,t\}^{\frac{1}{8}} \|f(t,\cdot)\|_{L^p}.
\end{equation*}

We now proceed to derive several estimates. For ease of reference, in what follows, we introduce  the notation:
\begin{equation*}
	Nu:=-div(u\otimes u).
\end{equation*}
\begin{proposition}\label{proposition1}
	It holds that
	\begin{equation*}
		\begin{split}
			\|Nu(t)\|_{\dot{H}^{-1}(\mathbb{R}^3)}
			&\leq \|U_0(t)\|_{\dot{H}^{-1}(\mathbb{R}^3)}
			+2\|u_L\|_{L^\infty(\mathbb{R}^3)} 
			\int_{0}^{t} \frac{1}{(t-\tau)^{\frac{1}{2}}}\|Nu(\tau)\|_{\dot{H}^{-1}(\mathbb{R}^3)} d\tau\\
			&\quad+\big(\int_{0}^{t} \frac{1}{(t-\tau)^{\frac{7}{8}}}\|Nu(\tau)\|_{\dot{H}^{-1}(\mathbb{R}^3)} d\tau\big)^2.
		\end{split}
	\end{equation*}
\end{proposition}
\begin{proof}
	Applying the Leray projection to both sides of the equation \eqref{NS}, we obtain
	\begin{equation}\label{NS1}
		\partial_t u-\Delta u=\mathbb{P}(Nu).
	\end{equation}
	We write the solution of \eqref{NS1} as follows:
	\begin{equation*}
		u(t,x)=u_L (t,x)+v(t,x),
	\end{equation*}	
	where
	$
	v(t,x)=\int_{0}^{t} e^{(t-\tau)\Delta}  \mathbb{P}(Nu(\tau)) d\tau
	$
	and $\mathbb{P}$ is the Leray projection on the divergence-free vector fields subspace. Thus we have
	\begin{equation*}
		Nu(t,x)
		=-div \left(
		(u_L+v) \otimes
		(u_L +v) 
		\right)
		=-U_0(t,x)-
		div (u_L\otimes v)- div (
		v \otimes u_L )
		- div (v \otimes v ).
	\end{equation*}	
	Denote by $K(t,x):=\frac{1}{(4\pi t)^{\frac{3}{2}}}e^{-\frac{|x|^2}{4t}}$ the heat kernel. By applying H\"{o}lder inequality, Young's inequality  and the propertied of the Leray projection, it becomes evident that
	\begin{equation*}
		\begin{split}
		\|Nu(t)\|_{\dot{H}^{-1}(\mathbb{R}^3)}
		&\leq \|U_0(t)\|_{\dot{H}^{-1}(\mathbb{R}^3)}
		+ 2\|u_L\|_{L^\infty(\mathbb{R}^3)} 
		\left\|v\right\|_{L^2(\mathbb{R}^3)}
		+\left\|v\right\|_{L^4(\mathbb{R}^3)}^2\\
		&\leq \|U_0(t)\|_{\dot{H}^{-1}(\mathbb{R}^3)}
		+2\|u_L\|_{L^\infty(\mathbb{R}^3)} 
		\int_{0}^{t} 
		\big\|(-\Delta)^{\frac{1}{2}} K(t-\tau,x-y) \big\|_{L^1(\mathbb{R}^3)}
		\|(-\Delta)^{-\frac{1}{2}}Nu(\tau)\|_{L^2(\mathbb{R}^3)} d\tau\\
		&\quad+\big(	\int_{0}^{t} 
		\big\|(-\Delta)^{\frac{1}{2}} K(t-\tau,x-y) \big\|_{L^{\frac{4}{3}}(\mathbb{R}^3)}
		\|(-\Delta)^{-\frac{1}{2}}Nu(\tau)\|_{L^2(\mathbb{R}^3)} d\tau\big)^2\\
		&\leq \|U_0(t)\|_{\dot{H}^{-1}(\mathbb{R}^3)}
		+2\|u_L\|_{L^\infty(\mathbb{R}^3)} 
		\int_{0}^{t} \frac{1}{(t-\tau)^{\frac{1}{2}}}\|Nu(\tau)\|_{\dot{H}^{-1}(\mathbb{R}^3)} d\tau\\
		&\quad+\big(\int_{0}^{t} \frac{1}{(t-\tau)^{\frac{7}{8}}}\|Nu(\tau)\|_{\dot{H}^{-1}(\mathbb{R}^3)} d\tau\big)^2,
		\end{split}
	\end{equation*}	
	which implies that we complete the proof.	
\end{proof}	

The next proposition is the main result of this section, which is instrumental in the proof of Theorem \ref{thm 1.1}.
\begin{proposition}\label{Tu}
	There exists $\epsilon_0 >0$ such that if the initial data 
	$
	\interleave (-\Delta)^{-\frac{1}{2}} U_0(t)\interleave_{2,\infty}
	\leq \epsilon_0
	$
	is satisfied, then the following estimate holds:
	\begin{equation}\label{|||Tu|||}
		\interleave (-\Delta)^{-\frac{1}{2}}Nu(t) \interleave_{2,\infty}
		\leq 2
		\interleave (-\Delta)^{-\frac{1}{2}}U_0(t) \interleave_{2,\infty}
		^{\frac{1}{2}}.
	\end{equation}
\end{proposition}
\begin{proof} 
	We employ a bootstrap argument to establish this proposition. Initially, we denote 
	\begin{equation*}
	\interleave (-\Delta)^{-\frac{1}{2}}U_0(t,x) \interleave_{2,\infty}
	^{\frac{1}{2}}=\epsilon\leq\epsilon_0.
	\end{equation*}
	
	We first consider the behavior for small times: $0\leq t\leq 1$.
	We assume that
	$	
	\underset{0\leq t\leq1}{\sup} t^{\frac{1}{8}}\|Nu(t)\|_{\dot{H}^{-1}(\mathbb{R}^3)}
	\leq 2 \epsilon^{\frac{1}{2}}.
	$
	Based on Proposition \ref{proposition1}, we conclude that
	\begin{equation*}
		\begin{split}
		t^{\frac{1}{8}}\|Nu(t)\|_{\dot{H}^{-1}(\mathbb{R}^3)}
		&\leq \epsilon
		+2t^{\frac{1}{8}}\|u_L\|_{L^\infty(\mathbb{R}^3)} \int_{0}^{t} \frac{1}{(t-\tau)^{\frac{1}{2}}\tau^{\frac{1}{8}}} \tau^{\frac{1}{8}}\|Nu(\tau)\|_{\dot{H}^{-1}(\mathbb{R}^3)} d\tau\\
		&\quad+t^{\frac{1}{8}}\big(\int_{0}^{t} \frac{1}{(t-\tau)^{\frac{7}{8}}\tau^{\frac{1}{8}}} \tau^{\frac{1}{8}}\|Nu(\tau)\|_{\dot{H}^{-1}(\mathbb{R}^3)} d\tau\big)
		\big(\int_{0}^{t} \frac{1}{(t-\tau)^{\frac{7}{8}}\tau^{\frac{1}{8}}} 2\epsilon^{\frac{1}{2}}d\tau\big)\\
		&\leq \epsilon
		+2t^{\frac{1}{8}}\|u_L\|_{L^\infty(\mathbb{R}^3)} \int_{0}^{t} \frac{1}{(t-\tau)^{\frac{1}{2}}\tau^{\frac{1}{8}}} \tau^{\frac{1}{8}}\|Nu(\tau)\|_{\dot{H}^{-1}(\mathbb{R}^3)} d\tau\\
		&\quad
		+20\epsilon^{\frac{1}{2}}
		t^{\frac{1}{8}}\big(\int_{0}^{t} \frac{1}{(t-\tau)^{\frac{7}{8}}\tau^{\frac{1}{8}}} \tau^{\frac{1}{8}}\|Nu(\tau)\|_{\dot{H}^{-1}(\mathbb{R}^3)} d\tau\big),
		\end{split}
	\end{equation*}	
	which, together with Lemma \ref{Lemma 1} and the choice of
	$a_0=\epsilon, C_1=2\|u_L\|_{L^\infty(\mathbb{R}^3)}=:M_0, C_2=20\epsilon^{\frac{1}{2}} $,
	implies that 
	\begin{equation*}
		\underset{0\leq t\leq 1}{\sup}	t^{\frac{1}{8}}\|Nu(t)\|_{\dot{H}^{-1}(\mathbb{R}^3)}
		\leq \epsilon e^{20^8(M_0+10)^8(20\epsilon^{\frac{1}{2}}+10)^8} 
		\leq \epsilon^{\frac{1}{2}},
	\end{equation*}
	where we set $\epsilon_0=\exp(-2\cdot26^8 20^8 T^* (M_0+10)^8)$.
	
	We now consider the behavior for large times: $t \geq 1$.
	We assume that
	$	
	\underset{1\leq t\leq T^*}{\sup}
	\|Nu(t)\|_{\dot{H}^{-1}(\mathbb{R}^3)}
	\leq 2 \epsilon^{\frac{1}{2}}.
	$
	Applying Proposition \ref{proposition1} once more, we deduce that
	\begin{equation*}
		\begin{split}
		\|Nu(t)\|_{\dot{H}^{-1}(\mathbb{R}^3)}
		&\leq \epsilon
		+2\|u_L\|_{L^\infty(\mathbb{R}^3)}
		\int_{0}^{t} \frac{1}{(t-\tau)^{\frac{1}{2}}}\|Nu(\tau)\|_{\dot{H}^{-1}(\mathbb{R}^3)} d\tau\\
		&\quad+\big(\int_{0}^{t} \frac{1}{(t-\tau)^{\frac{7}{8}}}\|Nu(\tau)\|_{\dot{H}^{-1}(\mathbb{R}^3)} d\tau\big)
		\big(\int_{0}^{t} \frac{1}{(t-\tau)^{\frac{7}{8}}}2\epsilon^{\frac{1}{2}} d\tau\big)\\
		&\leq \epsilon
		+2\|u_L\|_{L^\infty(\mathbb{R}^3)}
		\int_{0}^{t} \frac{1}{(t-\tau)^{\frac{1}{2}}}\|Nu(\tau)\|_{\dot{H}^{-1}(\mathbb{R}^3)} d\tau\\
		&\quad+16t^{\frac{1}{8}} \epsilon^\frac{1}{2}
		\int_{0}^{t} \frac{1}{(t-\tau)^{\frac{7}{8}}}\|Nu(\tau)\|_{\dot{H}^{-1}(\mathbb{R}^3)} d\tau.
		\end{split}
	\end{equation*}	
	Similarly, by applying Lemma \ref{Lemma 1} and choosing
	$a_0=\epsilon, C_1=2\|u_L\|_{L^\infty(\mathbb{R}^3)}=:M_0, C_2=16\epsilon^{\frac{1}{2}} $, we can arrive at
	\begin{equation*}
		\underset{1\leq t\leq T^*}{\sup} \|Nu(t)\|_{\dot{H}^{-1}(\mathbb{R}^3)}
		\leq \epsilon e^{20^8T^* (M_0+10)^8(16\epsilon^{\frac{1}{2}}+10)^8} 
		\leq \epsilon^{\frac{1}{2}}.
	\end{equation*}
	Similarly, we also set $\epsilon_0=\exp(-2\cdot26^8 20^8 T^* (M_0+10)^8)$.\\
	These considerations lead to the result stated in \eqref{|||Tu|||}. Consequently, the proof is completed.
\end{proof}

Now, we can present the following Proposition.
\begin{proposition}\label{prop 3}
	%For... with initial data $u_0\in L^2(\mathbb{R}^3)$, we have
	It holds that
	\begin{equation*}
		\begin{split}
			\|v(T)\|_{\dot{H}^1(\mathbb{R}^3)}^2
			+\int_{0}^{T} \|v(t)\|_{\dot{H}^2(\mathbb{R}^3)}^2 dt
			&\leq
			\exp\big(C \|u_0\|_{\dot{B}^0_{3,2}(\mathbb{R}^3)}^2 \big)
			\|\nabla v\|_{L^3([0,T],L_x^3(\mathbb{R}^3))}^3\\
			&\quad+4
			\exp\big(C \|u_0\|_{\dot{B}^0_{3,2}(\mathbb{R}^3)}^2 \big)
			\|U_0\|_{L^2([0,T],L_x^2(\mathbb{R}^3))}^2.
		\end{split}
	\end{equation*}
\end{proposition}
\begin{proof}
	The equation \eqref{NS} can be reformulated as
	\begin{equation}\label{NS2}
		\partial_t v -\Delta v +\nabla p=-v\cdot \nabla v -v \cdot \nabla u_L - u_L \cdot \nabla v -U_0.
	\end{equation}
	Multiply \eqref{NS2} by $-\Delta v$   and integrate over $dx$. Applying integration by parts yields
	\begin{equation*}
		\partial_t \int_{\mathbb{R}^3} |\nabla v|^2 dx
		+2\int_{\mathbb{R}^3} |\Delta v|^2 dx
		\leq \int_{\mathbb{R}^3} |\nabla v|^3 dx
		+2\int_{\mathbb{R}^3} |v| |\nabla u_L| |\Delta v| dx
		+\int_{\mathbb{R}^3} |\nabla u_L| |\nabla v|^2 dx
		+2\int_{\mathbb{R}^3} |U_0 \Delta v| dx.
	\end{equation*}
	Utilizing H\"{o}lder inequality and Young's inequality, we derive
	\begin{equation}\label{1}
		\begin{split}
			\int_{\mathbb{R}^3} |v| |\nabla u_L| |\Delta v| dx
			&\leq \frac{1}{8} \int_{\mathbb{R}^3} |\Delta v|^2 dx
			+ 2 \int_{\mathbb{R}^3} |v|^2 |\nabla u_L|^2 dx\\
			&\leq \frac{1}{8} \int_{\mathbb{R}^3} |\Delta v|^2 dx
			+2 \|v\|_{L^6(\mathbb{R}^3)} ^2
			\|\nabla u_L\|_{L^3(\mathbb{R}^3)}^2\\
			&\leq \frac{1}{8} \int_{\mathbb{R}^3} |\Delta v|^2 dx
			+2C \|\nabla v\|_{L^2(\mathbb{R}^3)}^2
			\|\nabla u_L\|_{L^3(\mathbb{R}^3)}^2,
		\end{split}
	\end{equation}
	where in the last line we make use of the fact that $\|v\|_{L^6(\mathbb{R}^3)}\leq C \|v\|_{\dot{H}^1(\mathbb{R}^3)}$.\\
	Note that
	\begin{equation*}
		\|\nabla v\|_{L^3(\mathbb{R}^3)} \leq C \|\Delta v\|_{L^2(\mathbb{R}^3)}^{\frac{1}{2}} 
		\|\nabla v\|_{L^2(\mathbb{R}^3)}^{\frac{1}{2}},
	\end{equation*}
	it readily follows that
	\begin{equation}\label{2}
		\begin{split}
			\int_{\mathbb{R}^3} |\nabla u_L| |\nabla v|^2 dx
			&\leq C\|\nabla u_L\|_{L^3(\mathbb{R}^3)}
			\|\nabla v\|_{L^3(\mathbb{R}^3)}^2\\
			&\leq \frac{1}{8} \|\Delta v\|_{L^2(\mathbb{R}^3)}^2
			+2C \|\nabla v\|_{L^2(\mathbb{R}^3)}^2
			\|\nabla u_L\|_{L^3(\mathbb{R}^3)}^2.
		\end{split}
	\end{equation}
	For similarity, in what follows, we denote $\mathcal{F}(t):=2\int_{\mathbb{R}^3} |U_0 \Delta v |dx
	+ \int_{\mathbb{R}^3} |\nabla v|^3 dx$. Thus
	the estimate \eqref{1} and the fact \eqref{2}
	collectively imply that
	\begin{equation*}
		\partial_t \|v(t)\|_{\dot{H}^1(\mathbb{R}^3)}^2
		+\frac{3}{2} \|v(t)\|_{\dot{H}^2(\mathbb{R}^3)}^2
		\leq C \|v(t)\|_{\dot{H}^1(\mathbb{R}^3)}^2
		\|u_L(t)\|_{\dot{W}^{1,3}(\mathbb{R}^3)}^2
		+\mathcal{F}(t).
	\end{equation*}
	As a result, it comes out
	\begin{equation*}
		\begin{split}
			&\partial_t 
			\big[
			\exp\big(-C \int_{0}^{t} \|u_L(\tau)\|_{\dot{W}^{1,3}(\mathbb{R}^3)}^2 d\tau\big)
			\|v(t)\|_{\dot{H}^1(\mathbb{R}^3)}^2
			\big]
			+\frac{3}{2} \exp\big(-C \int_{0}^{t} \|u_L(\tau)\|_{\dot{W}^{1,3}(\mathbb{R}^3)}^2 d\tau\big)
			\|v(t)\|_{\dot{H}^2(\mathbb{R}^3)}^2\\
			&\quad\leq \exp\big(-C \int_{0}^{t} \|u_L(\tau)\|_{\dot{W}^{1,3}(\mathbb{R}^3)}^2 d\tau\big)
			\mathcal{F}(t).
		\end{split}
	\end{equation*}
	Taking integral over time leads to the following inequality:
	\begin{equation*}
		\begin{split}
			&\exp\big(-C \int_{0}^{T} \|u_L(\tau)\|_{\dot{W}^{1,3}(\mathbb{R}^3)}^2 d\tau\big)
			\|v(T)\|_{\dot{H}^1(\mathbb{R}^3)}^2
			-\|v(0)\|_{\dot{H}^1(\mathbb{R}^3)}^2\\
			&\quad+\frac{3}{2} \int_{0}^{T} \exp\big(-C \int_{0}^{t} \|u_L(\tau)\|_{\dot{W}^{1,3}(\mathbb{R}^3)}^2 d\tau\big)
			\|v(t)\|_{\dot{H}^2(\mathbb{R}^3)}^2 dt\\
			&\quad\leq \int_{0}^{T}
			\exp\big(-C \int_{0}^{t} \|u_L(\tau)\|_{\dot{W}^{1,3}(\mathbb{R}^3)}^2 d\tau\big)
			\mathcal{F}(t) dt.
		\end{split}
	\end{equation*}
	Multiply both sides of the inequality by $\exp\big(C \int_{0}^{T} \|u_L(\tau)\|_{\dot{W}^{1,3}(\mathbb{R}^3)}^2 d\tau\big)$, which produces
	\begin{equation*}
		\begin{split}
		&\|v(T)\|_{\dot{H}^1(\mathbb{R}^3)}^2
		+ \frac{3}{2} \exp\big(C \int_{0}^{T} \|u_L(\tau)\|_{\dot{W}^{1,3}(\mathbb{R}^3)}^2 d\tau\big)
		\int_{0}^{T} \exp\big(-C \int_{0}^{t} \|u_L(\tau)\|_{\dot{W}^{1,3}(\mathbb{R}^3)}^2 d\tau\big)
		\|v(t)\|_{\dot{H}^2(\mathbb{R}^3)}^2 dt\\
		&\quad\leq \exp\big(C \int_{0}^{T} \|u_L(\tau)\|_{\dot{W}^{1,3}(\mathbb{R}^3)}^2 d\tau\big) \int_{0}^{T}
		\exp\big(-C \int_{0}^{t} \|u_L(\tau)\|_{\dot{W}^{1,3}(\mathbb{R}^3)}^2 d\tau\big)
		\mathcal{F}(t) dt,
		\end{split}
	\end{equation*}
	here $\|v(0)\|_{\dot{H}^1(\mathbb{R}^3)}=0$.
	From the preceding estimates, we can infer that
	\begin{equation}\label{estimate}
		\|v(T)\|_{\dot{H}^1(\mathbb{R}^3)}^2
		+\frac{3}{2}\int_{0}^{T} \|v(t)\|_{\dot{H}^2(\mathbb{R}^3)}^2 dt
		\leq \exp\big(C \int_{0}^{T} \|u_L(\tau)\|_{\dot{W}^{1,3}(\mathbb{R}^3)}^2 d\tau
		\big)
		\int_{0}^{T} \mathcal{F}(t) dt.
	\end{equation}
	Due to
	\begin{equation*}
		\int_{0}^{T} \|u_L(t)\|_{\dot{W}^{1,3}(\mathbb{R}^3)}^2 dt
		\leq \int_{0}^{+\infty} 
		\|u_L(t)\|_{\dot{W}^{1,3}(\mathbb{R}^3)}^2 dt
		\leq\|u_0\|_{\dot{B}^0_{3,2}(\mathbb{R}^3)}^2,
	\end{equation*}
	by substituting the above estimates into \eqref{estimate}, we deduce that
	\begin{equation}\label{m}
		\begin{split}
			&\|v(T)\|_{\dot{H}^1(\mathbb{R}^3)}^2
			+\frac{3}{2}\int_{0}^{T} \|v(t)\|_{\dot{H}^2(\mathbb{R}^3)}^2 dt\\
			&\quad\leq \exp\big(C \|u_0\|_{\dot{B}^0_{3,2}(\mathbb{R}^3)}^2 \big)
			\big(\int_{0}^{T} \int_{\mathbb{R}^3} |\nabla v|^3 dx dt
			+ 2\int_{0}^{T} \|U_0\|_{L^2(\mathbb{R}^3)} \|\Delta v\|_{L^2(\mathbb{R}^3)} dt
			\big) .
		\end{split}
	\end{equation}
	Notice that 
	\begin{equation*}
		\begin{split}
			&\int_{0}^{T} \|U_0\|_{L^2(\mathbb{R}^3)} \|\Delta v\|_{L^2(\mathbb{R}^3)} dt\\
			&\quad\leq  \int_{0}^{T} (2\exp\big(C \|u_0\|_{\dot{B}^0_{3,2}(\mathbb{R}^3)}^2 \big)
			\|U_0\|_{L^2(\mathbb{R}^3)}^2
			+\frac{1}{8\exp\big(C \|u_0\|_{\dot{B}^0_{3,2}(\mathbb{R}^3)}^2\big)}\|\Delta v\|_{L^2(\mathbb{R}^3)}^2) dt,
		\end{split}
	\end{equation*}
	absorbing the second term of right hand side by \eqref{m}, we arrive at
	\begin{equation*}
		\begin{split}
			&\|v(T)\|_{\dot{H}^1(\mathbb{R}^3)}^2
			+\int_{0}^{T} \|v(t)\|_{\dot{H}^2(\mathbb{R}^3)}^2 dt\\
			&\quad\leq
			\exp\big(C \|u_0\|_{\dot{B}^0_{3,2}(\mathbb{R}^3)}^2 \big)
			\|\nabla v\|_{L^3([0,T],L_x^3(\mathbb{R}^3))}^3
			+4
			\exp\big(C \|u_0\|_{\dot{B}^0_{3,2}(\mathbb{R}^3)}^2 \big)
			\|U_0\|_{L^2([0,T],L_x^2(\mathbb{R}^3))}^2.
		\end{split}
	\end{equation*}
	Therefore, we complete the proof of Proposition \ref{prop 3}.
\end{proof}	
\section{Proof of Theorem \ref{thm 1.1}}
With the above estimates, we are now equipped to finalize the proof of Theorem \ref{thm 1.1}.
\begin{proof}[Proof of Theorem \ref{thm 1.1}]			
	Now, we divide the proof into two parts. For the interval $[T^*,+\infty)$, 
	multiply \eqref{NS} by $ u$ and integrate over $dx$ we have $L^2$ energy inequality
	\begin{equation*}\label{energy inequality}
		\underset{t}{\sup}\|u(t,\cdot)\|_{L^2(\mathbb{R}^3)}^2+2 \int_{0}^{t}\|\nabla u(s,\cdot)\|_{L^2(\mathbb{R}^3)}^2 ds\leq\|u_0\|_{L^2(\mathbb{R}^3)}^2,
	\end{equation*}
	which leads to the conclusion that for $\forall ~T_0>0$, we have
	\begin{equation*}
		\frac{1}{2} T^* \underset{[\frac{T^*}{2},T^*]}{\inf} \|u(t)\|_{\dot{H}^1(\mathbb{R}^3)}^2
		\leq \int_{\frac{T^*}{2}}^{T^*} \|u(t)\|_{\dot{H}^1(\mathbb{R}^3)}^2 dt
		\leq \|u_0\|_{L^2(\mathbb{R}^3)}.
	\end{equation*}
	It is sufficient to yield that for $T^*=\frac{4\|u_0\|_{L^2(\mathbb{R}^3)}}{\sqrt{\epsilon_0}}$, we have
	\begin{equation*}
		\underset{[\frac{T^*}{2},T^*]}{\inf} \|u(t)\|_{\dot{H}^1(\mathbb{R}^3)}^2
		\leq \frac{2\|u_0\|_{L^2(\mathbb{R}^3)}}{T^*}
		= \frac{\epsilon_0^{\frac{1}{2}}}{2}.
	\end{equation*}
	As a result, there exist $T_0^*\in [\frac{T^*}{2},T^*]$ such that 
	$
	\|u(T_0^*)\|_{\dot{H}^1(\mathbb{R}^3)}^2 \leq \epsilon_0^{\frac{1}{2}}.
	$
	Applying Young's inequality, we arrive at
	\begin{equation*}
		\|u(T^*)\|_{\dot{H}^{\frac{1}{2}}(\mathbb{R}^3)}
		\lesssim 
		\|u(T^*)\|_{\dot{H}^1(\mathbb{R}^3)}^{\frac{1}{2}}
		\|u(T^*)\|_{L^2(\mathbb{R}^3)}^{\frac{1}{2}}
		\lesssim \epsilon_0^{\frac{1}{4}}
		\|u_0\|_{L^2(\mathbb{R}^3)}^{\frac{1}{2}},
	\end{equation*}
	where in the last inequality we use the fact that 
	$\|u(T^*)\|_{L^2(\mathbb{R}^3)} \leq \|u_0\|_{L^2(\mathbb{R}^3)}$.\\
	According to reference\cite{Fujita and Kato 1964}, it is obvious that there admits a unique global-in-time solution when $t \in [T_0^*,+\infty)$. 
	
	For the interval $[0,T^*]$, $T^*=\frac{4\|u_0\|_{L^2(\mathbb{R}^3)}}{\sqrt{\epsilon_0}}$, we employ a bootstrap argument to address the problem.
	Denote 
	$\|U_0\|_{L^2([0,T^*],L_x^2(\mathbb{R}^3))}=\epsilon \leq \epsilon_0(\|u_0\|_{\dot{B}^0_{3,2}(\mathbb{R}^3)})$.\\
	Assume
	\begin{equation*}
		\underset{0<t<T^*}{\sup} \|v(t)\|_{\dot{H}^1(\mathbb{R}^3)}^2
		+\int_{0}^{T^*} \|v(t)\|_{\dot{H}^2(\mathbb{R}^3)}^2 dt
		\leq 2C \epsilon^2
		\exp\big(3C\|u_0\|_{\dot{B}^0_{3,2}(\mathbb{R}^3)}^2\big).
	\end{equation*}
	By virtue of Gagliardo-Nirenberg interpolation, it is easy to find
	\begin{equation*}
		\|\nabla v\|_{L^3(\mathbb{R}^3)} \leq C
		\|v\|_{\dot{H}^2(\mathbb{R}^3)}^{\frac{2}{3}}
		\|v\|_{\dot{H}^1(\mathbb{R}^3)}^{\frac{1}{4}}
		\|v\|_{\dot{H}^{-1}(\mathbb{R}^3)}^{\frac{1}{12}},
	\end{equation*}
	thus by leveraging H\"{o}lder inequality, we immediately have
	\begin{equation*}
		\int_{0}^{T^*}\|\nabla v\|_{L^3(\mathbb{R}^3))}^3 dt
		\leq C
		\|v\|_{L_{T^*}^\infty \dot{H}^1(\mathbb{R}^3)}^{\frac{3}{4}}
		\|v\|_{L_{T^*}^\infty \dot{H}^{-1}(\mathbb{R}^3)}^{\frac{1}{4}}
		\int_{0}^{T^*} \|v(t)\|_{\dot{H}^2(\mathbb{R}^3)}^2 dt.
	\end{equation*}
	When $0<t<1$, on account of \eqref{|||Tu|||} we have
	\begin{equation}\label{t small}
		\|v\|_{L_{T^*}^\infty \dot{H}^{-1}(\mathbb{R}^3)}^{\frac{1}{4}}
		\leq 
		\underset{0\leq T \leq 1}{\sup}\big(
		\int_{0}^{T} \frac{2\interleave (-\Delta)^{-\frac{1}{2}} U_0(t)\interleave_{2,\infty}^{\frac{1}{2}}}{t^{\frac{1}{8}}} dt
		\big)^{\frac{1}{4}}
		\leq 2^{\frac{1}{2}}
		\interleave (-\Delta)^{-\frac{1}{2}} U_0(t)\interleave_{2,\infty}^{\frac{1}{8}}
		\leq 2\epsilon_0^{\frac{1}{8}}.
	\end{equation}
	When $1\leq t \leq T^*$, we have 
	\begin{equation}\label{t large}
		\begin{split}
			\|v\|_{L_T^\infty\dot{H}^{-1}(\mathbb{R}^3)}^{\frac{1}{4}} 
			&\leq \underset{0\leq T \leq T^*}{\sup}\big(\int_{0}^{T} \|Nu\|_{\dot{H}^{-1}}dt\big)^{\frac{1}{4}}\\
			&\leq \big(\underset{[0,1]}{\sup} \int_{0}^{1} \frac{t^{\frac{1}{8}}\|Nu\|_{\dot{H}^{-1}}}{t^{\frac{1}{8}}} dt+\underset{[1,T^*]}{\sup} \int_{1}^{T} \|Nu\|_{\dot{H}^{-1}} dt\big)^{\frac{1}{4}}\\
			&\leq\big(4\epsilon_0^{\frac{1}{2}} + 2\epsilon_0 T^*\big)^{\frac{1}{4}}.
		\end{split}
	\end{equation}
	Combining \eqref{t small} and \eqref{t large}, we achieve
	\begin{equation*}
		\begin{split}
		&\exp\big(C \|u_0\|_{\dot{B}^0_{3,2}(\mathbb{R}^3)}^2 \big)
		\|\nabla v\|_{L^3([0,T],L_x^3(\mathbb{R}^3))}^3\\
		&\leq
		\exp\big(C \|u_0\|_{\dot{B}^0_{3,2}(\mathbb{R}^3)}^2 \big) C
		(2C)^{\frac{3}{8}} \epsilon^{\frac{3}{4}} \big(\exp\big(3C\|u_0\|_{\dot{B}^0_{3,2}(\mathbb{R}^3)}^2\big)\big)^{\frac{3}{8}}
		2 C \epsilon^2 \exp\big(3C\|u_0\|_{\dot{B}^0_{3,2}(\mathbb{R}^3)}^2\big)
		\big( 4\epsilon_0^{\frac{1}{2}} +2\epsilon_0 T^*\big)^{\frac{1}{4}}\\
		&\leq C \epsilon^2 \exp(3C\|u_0\|_{\dot{B}^0_{3,2}}^2(\mathbb{R}^3)).
		\end{split}
	\end{equation*}
	On account of Proposition \ref{prop 3}, since $\epsilon_0$ is arbitrarily small, there exists $\epsilon_0$ such that
	\begin{equation*}
		\underset{0<t<T^*}{\sup} \|v(t)\|_{\dot{H}^1(\mathbb{R}^3)}^2
		+\int_{0}^{T^*} \|v(t)\|_{\dot{H}^2(\mathbb{R}^3)}^2 dt
		\leq C \epsilon^2 \exp\big(3C\|u_0\|_{\dot{B}^0_{3,2}(\mathbb{R}^3)}^2\big).
	\end{equation*}
\end{proof}

\textbf{Acknowledgements:} The authors are grateful to Professor Quoc-Hung Nguyen, who introduced this project to us and patiently guided, supported, and encouraged us during this work.


\begin{thebibliography}{99}
	\bibitem{P. Auscher}
	P. Auscher, S. Dubois and P. Tchamitchian, On the stability of global solutions to Navier–Stokes equations in the space, \emph{J. Math. Pures Appl.}, \textbf{83} (2004), 673-697.
	https://doi.org/10.1016/j.matpur.2004.01.003.
	
	\bibitem{Bahouri}
	 H. Bahouri, J.-Y Chemin and R. Danchin, Fourier Analysis and Nonlinear Partial Differential Equations, \emph{In Grundlehren der mathematischen Wissenschaften}, Springer Berlin Heidelberg, (2011).	
	 
	 \bibitem{H. Beirão da Veiga}
	 H. Beirão da Veiga and P. Secchi, $L^p$-stability for the strong solutions of the Navier–Stokes equations in the whole space, \emph{Arch. Rational Mech. Anal}, \textbf{98} (1987), 65–69. https://doi.org/10.1007/BF00279962.
	 
	 \bibitem{J. Bourgain}
	 J. Bourgain and N. Pavlović, Ill-posedness of the Navier–Stokes equations in a critical space in 3D, \emph{Journal of Functional Analysis},
	 \textbf{255} (2008), 2233-2247.
	 
	 \bibitem{Cannone}
	 M. Cannone, A generalization of a theorem by Kato on Navier-Stokes equations, \emph{Revista matemática iberoamericana}, \textbf{13} (1997), 515-541.
	 
	 \bibitem{M. Cannone}
	 M. Cannone, Chapter 3 - Harmonic Analysis Tools for Solving the Incompressible Navier–Stokes Equations,
	 \emph{Handbook of Mathematical Fluid Dynamics}, North-Holland, \textbf{3} (2005), 161-244.
	 https://doi.org/10.1016/S1874-5792(05)80006-0.
	
	\bibitem{Y. Meyer}
	M. Cannone, Y. Meyer and F. Planchon, Solutions auto-similaires des équations de Navier-Stokes,
	\emph{Séminaire Équations aux dérivées partielles (Polytechnique)}, (1993-1994), 1-10. http://eudml.org/doc/112096.
	
	
	\bibitem{J.Y. Chemin}
	 J.-Y. Chemin and  I. Gallagher, On the global wellposedness of the 3-D Navier–Stokes equations with large initial data, \emph{Ann. Sci. École Norm. Sup.}, \textbf{39} (2006), 679-698.
	
	\bibitem{I. Gallagher}
    J.-Y. Chemin and  I. Gallagher, Wellposedness and stability results for the Navier–Stokes equations in $\mathbb{R}^3$, \emph{Ann. I. H. Poincaré – AN}, \textbf{26} (2009), 599-624. https://doi.org/10.1016/j.anihpc.2007.05.008.
    
    \bibitem{G. Ciampa}
    G. Ciampa and R. Lucà,
    Localization of Beltrami fields: Global smooth solutions and vortex reconnection for the Navier-Stokes equations,
    \emph{Journal of Functional Analysis},
    \textbf{287} (2024), 110610.
    https://doi.org/10.1016/j.jfa.2024.110610.
    
    \bibitem{clay}
    C. L. Fefferman,
    https://www.claymath.org/millennium/navier-stokes-equation/, Accessed 1 May, 2000.
	
	\bibitem{M. Fujii}
	M. Fujii, 
	Global solutions to the rotating Navier–Stokes equations with large data in the critical Fourier–Besov spaces,
	\emph{Math. Nachr.},
	\textbf{297} (2024), 1678-1693.
	 https://doi.org/10.1002/mana.202300226.
	
	
	\bibitem{Fujita and Kato 1964} 
	H. Fujita and T. Kato, On the Navier–Stokes initial value problem. I, \textit{Arch. Rational Mech. Anal. }16 (1964), 269–315. 
	
	\bibitem{T. Miyakawa}
	Y. Giga and T. Miyakawa, Solutions in $L^r$ of the Navier-Stokes initial value problem, \emph{Arch. Ration.
	Mech. Anal.}, \textbf{89} (1985), 267–281. https://doi.org/10.1007/BF00276875.
	
	\bibitem{Y. Giga}
	Y. Giga and T. Miyakawa, Navier-stokes flow in $r3$
	with measures as initial vorticity and morrey spaces, \emph{Communications in Partial Differential Equations}, \textbf{14} (1989), 577-618. https://doi.org/10.1080/03605308908820621.
	
	\bibitem{T. Kato}
	T. Kato, Strong $L^p$-solutions of the Navier-Stokes equation in $R^m$, with applications to weak solutions, \emph{Math. Z.}, \textbf{187} (1984), 471–480. https://doi.org/10.1007/BF01174182.
	
	\bibitem{H. Koch}
	H. Koch and D. Tataru, Well-posedness for the Navier–Stokes equations,  \emph{Adv. Math.}, \textbf{157} (2001), 22-35.
	
	\bibitem{Leray} J. Leray, Sur le mouvement d’un liquide visqueux emplissant l’espace. Acta Math. 63 no. 1, 193-248 (1934).
	
	\bibitem{F. Planchon}
	F. Planchon, Global strong solutions in Sobolev or Lebesgue spaces to the incompressible Navier-Stokes equations in $\mathbb{R}^3$, \emph{Ann. I. H. Poincaré – AN}, \textbf{13} (1996), 319-336.
	
	\bibitem{G. Ponce}
	G. Ponce, R. Racke and T. C. Sideris et al, Global stability of large solutions to the 3D Navier-Stokes equations, \emph{Commun.Math. Phys.}, \textbf{159} (1994), 329–341. https://doi.org/10.1007/BF02102642.
	
	\bibitem{Michael E. Taylor}
	M. E. Taylor, Analysis on Morrey Spaces and Applications to Navier-Stokes and Other
	Evolution Equations, \emph{Communications in Partial Differential Equations}, \textbf{17} (1992), 1407-1456. https://doi.org/10.1080/03605309208820892.
	
	\bibitem{S. Wang}
	S. Wang, Z. Zhang and Y. Zhou,
	Remarks on the Global Existence for Incompressible Navier-Stokes Equations,
	\emph{Chin. Ann. Math. Ser. B}, \textbf{45} (2024), 529-536.
	https://doi.org/10.1007/s11401-024-0027-3.
\end{thebibliography}
\end{document}